\documentclass[12pt]{article}
\input epsf.tex

\usepackage{graphicx}
\usepackage{amsmath,amsthm,amsfonts,amscd,amssymb,comment,eucal,latexsym,mathrsfs}
\usepackage{stmaryrd}
\usepackage[all]{xy}

\usepackage{epsfig}

\usepackage[all]{xy}
\xyoption{poly}
\usepackage{fancyhdr}
\usepackage{wrapfig}
\usepackage{epsfig}
\usepackage{hyperref}

\theoremstyle{plain}

\newtheorem{maintheorem}{Theorem}
\newtheorem{thm}{Theorem}[section]
\newtheorem{prop}[thm]{Proposition}
\newtheorem{lem}[thm]{Lemma}

\theoremstyle{definition}

\theoremstyle{remark}

\advance \topmargin by -\headheight
\advance \topmargin by -\headsep
\evensidemargin \oddsidemargin
\def\cc{{\curvearrowright}}

\def\N{{\mathbb N}}

\def\R{{\mathbb R}}

\def\chix{{\raise.5ex\hbox{$\chi$}}}

\begin{document}
\title{Property (T) and the Furstenberg Entropy of Nonsingular
  Actions}

  \author{Lewis Bowen\footnote{University of Texas at
    Austin. Supported in part by NSF grant DMS-0968762, NSF CAREER
    Award DMS-0954606 and BSF grant 2008274.}, Yair
  Hartman\footnote{Weizmann Institute of Science. Supported by the
    European Research Council, grant 239885.} and Omer
  Tamuz\footnote{Microsoft Research.}}
  
\maketitle

\begin{abstract}
  We establish a new characterization of property (T) in terms of the
  Furstenberg entropy of nonsingular actions. Given any generating
  measure $\mu$ on a countable group $G$, A.\ Nevo showed that a
  necessary condition for $G$ to have property (T) is that the
  Furstenberg $\mu$-entropy values of the ergodic, properly
  nonsingular $G$-actions are bounded away from zero. We show that
  this is also a sufficient condition.
\end{abstract}

%
%
%

\section{Introduction}

A measurable action on a probability space $G \cc (X,\eta)$ is called
a {\em nonsingular action} if the measure $\eta$ is quasi-invariant
with respect to any $g \in G$; that is, if $\eta$ and $g_*\eta$ are
equivalent (i.e., mutually absolutely continuous) measures. We say
that the action is {\em properly nonsingular} if $\eta$ is not
equivalent to a $G$-invariant probability measure.

A probability measure $\mu$ on a countable group $G$ is {\em
  generating} if its support generates $G$ as a semigroup.

Let $G \cc (X,\eta)$ be a nonsingular action and let $\mu$ be a
probability measure on $G$. The {\em Furstenberg
  entropy}~\cite{furstenberg1963noncommuting} or {\em $\mu$-entropy}
is defined by
$$h_\mu(X,\eta)=\sum_{g \in G}\mu(g)\int_X -\log
\frac{d g^{-1}_*\eta}{d\eta}(x)~d\eta(x).$$

Jensen's inequality implies that $h_\mu(X,\eta)\ge 0$ and that, for
generating measures, equality holds if and only if $\eta$ is an
invariant measure. Furstenberg entropy is an important conjugacy
invariant of nonsingular actions~\cite{nevo2000rigidity}, and in
particular of {\em stationary} actions; $G \cc (X,\eta)$ is
$\mu$-stationary if $\mu * \eta = \eta$.

The pair $(G,\mu)$ is said to have an {\em entropy gap} if there
exists some constant $\epsilon = \epsilon(G,\mu) >0$ such that the
$\mu$-entropy of any ergodic, properly nonsingular $G$-action on a
probability space is at least $\epsilon$.

A group $G$ has {\em property (T)}, if any unitary representation of
$G$ which has almost invariant vectors admits a non-zero invariant
vector.  The purpose of this paper is to establish a similar
characterization of property (T), using nonsingular actions instead of
unitary representations, and thinking of those with small entropy
values as being ``almost invariant''.

\begin{maintheorem}\label{thm:main}
  Let $G$ be a countable group. Then $G$ has property (T) if and only
  if for every (equivalently, for some) generating measure $\mu$,
  $(G,\mu)$ has an entropy gap.
\end{maintheorem}

For the case that $G$ has property (T) this theorem follows from
Nevo~\cite[Propositions 4.1, 4.2]{Ne03}.  Note that the statements of
these propositions omit the requirements that $\eta$ is ergodic and
properly nonsingular; the result is no longer true without either of
these hypotheses.  For completeness, we include in
Appendix~\ref{sec:nevo} a reproduction of Nevo's proof of this
direction.

The idea of the proof of the other direction (i.e., that groups
without property (T) have no entropy gap) is the following. We first
consider the group $2^\N_{fin}$; the elements of this group are the
finite subsets of $\N$ and the group operation is symmetric
difference. For this group, we construct in Section~\ref{sec-2^N} very
simple nonsingular actions which show that $2^\N_{fin}$ has no entropy
gap, for a family of natural measures on $2^\N_{fin}$.
  
Then, in Section~\ref{sec:cocycle}, we observe that any countable
group $G$ without property (T) admits a non-trivial cocycle into
$2^\N_{fin}$. Using these cocycles one can construct skew-product
$G$-actions over $2^\N_{fin}$ spaces; the entropy formula for
skew-product actions is established in
Section~\ref{sec:entropy-of-skew}.  Finally, in
Section~\ref{sec:proof} we prove Theorem~\ref{thm:main}, by showing
that the $2^\N_{fin}$-actions that we construct in
Section~\ref{sec-2^N}, and which have arbitrarily low Furstenberg
entropy, can be lifted to skew-product $G$-actions with arbitrarily
low Furstenberg entropy.

\subsection{Related literature}
\subsubsection{Characterizing property (T)}
Recently, Ozawa~\cite{ozawa2013noncommutative} showed that property
(T) can be characterized by the spectral gap of $\Delta_\mu \in
\R[G]$, the Laplacian associated with a finitely supported, symmetric
generating measure $\mu$. This characterization is one of many which
have been established since Kazhdan's first definition of property
(T)~\cite{kazhdan67connection}, some of which have been instrumental
in proving property (T) for groups that were previously not known to
have it. We refer the reader to Bekka, de La Harpe and
Valette~\cite{bekka2008kazhdan} for a complete discussion.

\subsubsection{Stationary actions}
A particularly interesting class of nonsingular actions are the
stationary ones. Furstenberg entropy has been a useful tool in the
study of stationary actions (e.g.,~\cite{furstenberg1963noncommuting,
  kaimanovich1983random}), and the study of the set of entropy values
realizable by properly nonsingular stationary actions has attracted
some interest~\cite{nevo2000rigidity, bowen2010random,
  hartman2012furstenberg}.

Nevo's theorem (Theorem~\ref{thm:nevo}) implies that for every
$(G,\mu)$, where $G$ has property (T), the Furstenberg entropy values
of properly nonsingular, ergodic $\mu$-stationary actions are bounded
away from zero; in this case we say that $(G,\mu)$ has a {\em
  stationary entropy gap}. In~\cite{bowen2010random}
and~\cite{hartman2012furstenberg} it is shown that some $(G,\mu)$
without property (T) have no stationary entropy gap. However, this is
not a characterization of property (T): in a previous
paper~\cite[Proposition 7.5]{Bo14} we show that there exist $(G,\mu)$
without property (T), but with a stationary entropy gap.

A question that therefore remains open is that of characterizing the
pairs $(G,\mu)$ that have a stationary entropy gap. More narrowly, it
is not even known that all amenable groups have no stationary entropy
gap.

\section{The case of $2^\N_{fin}$}\label{sec-2^N}
Let $2^\N_{fin}$ denote the set of all finite subsets of the natural
numbers $\N=\{1,2,3,\ldots\}$. Endowed with the operation of symmetric
difference, $2^\N_{fin}$ is an abelian group in which every element
other than the identity is an involution. For $T \in 2^\N_{fin}$, let
$\max(T) = \max_{t\in T}t$. We also let $\max(\emptyset) = 0$. Because $2^\N_{fin}$ is an abelian group we will use additive notation when expressing group multiplication. Thus $T+S:=T\vartriangle S=(T\setminus S) \cup (S\setminus T)$ for any $T,S \in 2^\N_{fin}$.

Let $2^\N_{fin}$ act on $2^\N$ by symmetric difference, identifying
elements of $2^\N$ with subsets of $\N$.  Let $\omega_p$ be a measure
on $2^\N$ given by $\omega_p=\prod_\N B(p)$, where $B(p)$ is the
Bernoulli $p$ measure on $\{0,1\}$.

\begin{prop}
  \label{prop:Z}
  Consider the action $2^\N_{fin} \cc (2^\N,\omega_p)$. Then, for any
  probability measure $\mu$ on $2^\N_{fin}$, the $\mu$-entropy of
  $(2^\N, \omega_p)$ is
  \begin{align*}
    h_\mu(2^\N, \omega_p) = \varphi(p)\sum_{T \in 2^\N_{fin}}\mu(T)|T| \leq   \varphi(p)\sum_{T \in 2^\N_{fin}}\mu(T)\max(T),
  \end{align*}
  where $\varphi(p) = \log\frac{p}{1-p}(2p-1)$.
\end{prop}
Note that $\lim_{p \to 1/2}\varphi(p) = 0$ and that for $p \neq 1/2$
it holds that $\varphi(p) > 0$.

\begin{proof}  
  For $T \in 2^\N_{fin}$, $T^{-1}_*\omega_p =T_*\omega_p= \prod_{n \in
    T} B(1-p)\prod_{n \not\in T} B(p)$.  Hence
  \begin{align*}
    \int_X -\log \frac{dT_*\omega}{d\omega}(x)~d\omega(x) =
    \varphi(p)|T|
  \end{align*}
  where $\varphi(p) = \log\frac{p}{1-p}(2p-1)$ is the Kullback-Leibler
  divergence between $B(p)$ and $B(1-p)$. Finally, it follows that
  for any measure $\mu$ on $2^\N_{fin}$, 
  \begin{align*}
    h_\mu(2^\N,\omega_p) = \varphi(p)\sum_{T \in 2^\N_{fin}}\mu(T)|T| \leq
    \varphi(p)\sum_{T \in 2^\N_{fin}}\mu(T)\max(T).
  \end{align*}
\end{proof}

It is easy to show that $\omega_p$ is ergodic and is not equivalent to
an invariant probability measure (see Lemma~\ref{lemma:converse-lln};
an invariant measure $\lambda$ would have to satisfy $\lambda(\{x
\,:\, x_n=1\}) = 1/2$). Hence it follows from this proposition that
$2^\N_{fin}$ does not have an entropy gap, for any $\mu$ with $\sum_{T
  \in 2^\N_{fin}}\mu(T)\max(T)<\infty$. In the remainder of this paper
we lift this result to any group without property (T).

\section{Non-property (T) and cocycles}\label{sec:cocycle}

Before stating the next proposition, let us recall some standard definitions. Given a nonsingular action $G \cc (X,\eta)$ and a countable group $\Gamma$, a {\em cocycle} $c:G \times X \to \Gamma$ is a measurable map such that 
$$c(gh,x) = c(g,hx)c(h,x)$$
for a.e. $x$ and every $g,h \in G$. Two such cocycles $c_1,c_2:G \times X \to \Gamma$ are {\em cohomologous} if there exists a measurable function $\beta:X \to \Gamma$ such that
$$c_1(g,x) = \beta(gx)c_2(g,x)\beta(x)^{-1}$$
for every $g\in G$ and a.e.\ $x$. A cocycle is a {\em coboundary} if
it is cohomologous to the trivial cocycle (whose essential image is
contained in the trivial element $\{e\}$).

\begin{prop}\label{prop:key}
  Assume $G$ does not have property (T). Let $\rho:G \to \N$
  be a proper function (so $\rho^{-1}(n)$ is finite for every $n \in
  \N$). Then there exists an ergodic probability measure preserving
  action $G \cc (X,\eta)$ and a Borel cocycle $c:G \times X
  \to 2^\N_{fin}$ such that $c$ is not cohomologous to a cocycle
  with a finite image.  Moreover, for every element $g \in G$,
    \begin{align*}
       \int_X \max(c(g,x))~d\eta(x) \le \rho(g).
    \end{align*}
\end{prop}

\begin{proof}
  This result follows from the proof of~\cite[Theorem 2.1]{JS87} as we
  now explain. It is a well-known result of Connes and
  Weiss~\cite{connes80property} that if $G$ does not have property (T)
  then there exists an ergodic probability measure preserving action
  $G \cc (X,\eta)$ which is not strongly ergodic (see
  also~\cite[Theorem 6.3.4]{bekka2008kazhdan}).  By~\cite[Lemma
  2.4]{JS87} there exist Borel sets $D_n \subset X$ such that the
  following hold.
  \begin{enumerate}
  \item $D_n$ is asymptotically invariant: $\lim_{n\to\infty} \eta(
    gD_n \vartriangle D_n) = 0$ for every $g \in G$;
  \item $\lim_{n\to\infty}\eta(D_n) = 1/2$.
  \item For every $g \in G$ with $\rho(g)\le n$, $\eta(D_n
    \vartriangle gD_n) < 2^{-n}$.
  \end{enumerate}
  To be precise, the last statement follows from an easy modification
  of the proof of \cite[Lemma 2.4]{JS87}; see equation (2.10) there.

  Define a map $\phi \colon X \to 2^\N$ by $\phi(x) = \{j \in \N\,:\, x
  \in D_j\}$. We let $2^\N_{fin}$ act on $2^\N$ by symmetric
  difference.

  Let $g\in G$.  If $m \ge \rho(g)$ then
  $$\eta(\{x \in X\,:\, ~\phi(gx) \notin 2^\N_{fin} \vartriangle \phi(x)\}) \le \sum_{k=m}^\infty \eta(D_k \vartriangle gD_k)<2^{-m+1}.$$
  Since $m \ge \rho(g)$ is arbitrary, the left hand side must equal
  zero. Therefore $\phi(gx) \in 2^\N_{fin}\vartriangle \phi(x)$ for
  $\eta$-a.e.\ $x\in X$.

  Define the cocycle $c:G\times X \to 2^{\N}_{fin}$ by $c(g,x) = T$ if $\phi(gx)=T \vartriangle
  \phi(x)$. If $\rho(g)=m$ then
  \begin{eqnarray*}
    \int_X \max(c(g,x))~d\eta(x) &\le& m -1+  \int_{\{x:\max(c(g,x)) \ge m\}} \max(c(g,x))~d\eta(x)\\
    &\le& m -1+ \sum_{n=m}^\infty \eta(\{x\in X:~\max(c(g,x)) \ge n \}) \\
    &\le& m -1 + \sum_{n=m}^\infty \eta( D_n \vartriangle g D_n) \le m -1 + 2^{-m+1} \le m.
  \end{eqnarray*}

  The cocycle $c$ cannot be cohomologous to a cocycle with a finite
  image, since $\phi$ is nontrivial. To be precise, for $m\in \N$, we let $2^m < 2^\N_{fin}$ denote the subgroup consisting of all elements $T \in 2^\N_{fin}$ with support in $[m]:=\{1,\ldots, m\}$. To obtain a contradiction, suppose there is a
  Borel map $f \colon X \to 2^\N_{fin}$ and a cocycle $b \colon G
  \times X \to 2^m$ such that
  $$c(g,x) = f(gx)+b(g,x)+f(x)$$
  for a.e.\ $x \in X$ and every $g\in G$. Note that all elements
  of $2^\N_{fin}$ are involutions, and so $-f(x)=f(x)$.

  Let $\psi(x)=f(x)\vartriangle\phi(x)$. Observe that
 \begin{eqnarray*}
 \psi(gx) &=& f(gx)\vartriangle\phi(gx) = f(gx)\vartriangle c(g,x)\vartriangle \phi(x) = f(gx)\vartriangle f(gx)\vartriangle b(g,x)\vartriangle f(x)\vartriangle \phi(x) \\
 &=&b(g,x)\vartriangle \psi(x)
 \end{eqnarray*}
  for a.e.\ $x\in X$ and every $g\in G$.

  Let $\tilde \psi(x) = \psi(x)
  \setminus [m]$, $\tilde f(x) = f(x) \setminus [m]$ and $\tilde
  \phi(x) = \phi(x) \setminus [m]$. Then, since $b(g,x) \subseteq
  [m]$,
  $$\tilde\psi(gx) = (b(g,x)\vartriangle \psi(x)) \setminus [m] = \tilde\psi(x),$$
  and $\tilde\psi$ is $G$-invariant. Because $G \cc
  (X,\eta)$ is ergodic, there is an element $R \in 2^{\N \setminus [m]}$ such
  that $\tilde\psi(x)=R$ for a.e.\ $x$. Thus
  $$R=\tilde\psi(x)=\tilde f(x)\vartriangle  \tilde\phi(x),$$
  which implies $\tilde\phi(x) \vartriangle R = \tilde f(x) \in
  2^\N_{fin}$ for a.e.\ $x \in X$.

  Let $D'_n = D_n$ if $n \not \in R$, and otherwise let $D'_n = X
  \setminus D_n$, the complement of $D_n$. Then, for a.e.\ $x \in X$,
  there are only a finite number of elements $j \in \N \setminus [m]$
  such that $x \in D_j'$; these are precisely the elements in $\tilde
  f(x)$.

  It follows that if we let $A_n = \{x \in X \,:\, \max(\tilde f(x)) \geq
  n\}$, then $\lim_n\eta(A_n) = 0$. But $A_n$ contains $D_n'$, and so
  $\eta(A_n) \geq \eta(D_n')$, in contradiction to the fact that
  $\lim_n\eta(D_n') = 1/2$.

\end{proof}

\section{Entropy for skew-products}\label{sec:entropy-of-skew}

\begin{lem}\label{lem:entropy}
  Let $G$ be a countable group with a probability measure $\mu$. Let
  $G \cc (X,\eta)$ be a probability measure preserving action, and
  let $c:G \times X \to \Gamma$ be a cocycle to a countable group
  $\Gamma$. Also, let $\Gamma\cc (W,\omega)$ be a nonsingular action
  on the probability space $(W,\omega)$.  Let $G \cc X \times W$ be
  the skew-product action
  $$g(x,w) = (gx, c(g,x)w).$$
  Then
  $$h_\mu(X\times W,\eta \times \omega) = \int_X h_{\mu_x}(W,\omega)~d\eta(x)$$
  where $\mu_x$ is the pushforward of $\mu$ under the map $g \mapsto
  c(g,x)$.
\end{lem}

\begin{proof}
  \begin{eqnarray*}
    h_\mu(X\times W, \eta \times \omega) &=& 
    -\sum_{g\in G} \mu(g) \int_{X\times W}  \log \left( \frac{d g_*^{-1} (\eta \times \omega)}{d\eta \times \omega}(x,w)\right)~d\eta(x)d\omega(w)\\
    &=& -\sum_{g\in G} \mu(g) \int_{X\times W}  \log \left( \frac{d c(g,x)_*^{-1}\omega}{d \omega}(w)\right)~d\eta(x)d\omega(w)\\
    &=& -\int _X \sum_{\gamma \in \Gamma}  \mu_x(\gamma) \int_W \log \left( \frac{d \gamma_*^{-1}\omega}{d \omega}(w)\right)~d\omega(w) d\eta(x)\\
    &=&\int_X h_{\mu_x}(W,\omega)~d\eta(x).
  \end{eqnarray*}
\end{proof}
 
\begin{lem}
  \label{lemma:converse-lln}
  Let $E_n=\{x \in 2^\N: x_n=1\}$. Let $\nu$ be a Borel probability
  measure on $2^\N$ such that $\nu(E_n) \leq 1/2$ for infinitely many
  $n \in \N$, and let $\omega_p$ be the i.i.d.\ Bernoulli $p$ measure
  on $2^\N$, with $p > 1/2$. Then $\nu$ is not absolutely continuous
  with respect to $\omega_p$.
\end{lem}
\begin{proof}
  Let $\{i_k\}_{k = 1}^\infty$ be a sequence of indices such that
  $\nu(E_{i_k}) < 1/2$. For $x\in 2^\N$, denote
  \begin{align*}
    S_n(x) = \frac{1}{n}\sum_{k=1}^n1_{E_{i_k}}(x),
  \end{align*}
  and let
  \begin{align*}
    S(x) = \liminf_nS_n(x),
  \end{align*}
  so that, clearly, $\omega_p(\{x \in 2^\N\,:\,S(x) \neq p\}) = 0$. But
  \begin{align*}
  \int_X  S_n(x)~d\nu (x) = \frac{1}{n}\sum_{k=1}^n\nu(E_{i_k})\leq 1/2,
  \end{align*}
  and so by Fatou's Lemma,
  \begin{align*}
   \int_X  S(x)~d\nu (x) \leq \liminf_n\frac{1}{n}\sum_{k=1}^n\nu(E_{i_k}) \leq 1/2 <p.
  \end{align*}
  Hence $\nu(\{x \in 2^\N\,:\,S(x) \neq p\}) > 0$, and $\nu$
  is not absolutely continuous with respect to $\omega_p$.
\end{proof}

\begin{lem}
  \label{lem:invariant}
  Let $G \cc (X,\eta)$ be an ergodic probability measure preserving
  action, let $c:G \times X \to 2^\N_{fin}$ be a cocycle, and let
  $2^\N_{fin}$ act on $2^\N$ by symmetric difference. Consider the
  skew-product action $G \cc X \times 2^\N$ given by $g(x,y)=(gx,
  c(g,x)y)$.

  Denote by $\omega_p$ the Bernoulli $p$ i.i.d.\ measure on $2^\N$.
  For any $1/2<p<1$, if there exists a $G$-invariant probability
  measure that is absolutely continuous with respect to $\eta \times
  \omega_p$, then $c$ is cohomologous to a cocycle with a finite
  image.
\end{lem}
\begin{proof}
  Let $\lambda \ll \eta \times \omega_p$ be a $G$-invariant probability 
  measure. Note that
  $\lambda$ projects to a $G$-invariant probability measure on $X$
  which is absolutely continuous with respect to $\eta$. Because $\eta$
  is ergodic, this projection must equal $\eta$. Hence
  \begin{align*}
    d\lambda(x,y) = d\eta(x)d\lambda_x(y)
  \end{align*}
  where $\lambda=\int \delta_x\times \lambda_x~d\eta(x)$ is the disintegration of $\lambda$ over $\eta$ (here $\delta_x$ is the Dirac measure concentrated on $\{x\}$). Also
  \begin{align*}
    \frac{d\lambda}{d\eta\times\omega_p}(x,y) =
    \frac{d\lambda_x}{d\omega_p}(y).
  \end{align*}
  It follows that $\lambda_x \ll \omega_p$, for a.e.\ $x$.  By the
  invariance of $\lambda$ we have that, for every $g \in G$,
  $g_*\lambda = \lambda$ which implies 
  $$\lambda_{gx}=(g_*\lambda)_{gx} = c(g,x)_*\lambda_x$$
  for a.e. $x$.

  Let $E_n=\{x \in 2^\N: x_n=1\}$, and let $f(x) = \{n \in \N :
  \lambda_x(E_n)=1/2\}$. By Lemma~\ref{lemma:converse-lln} $f(x)$ is
  finite for a.e.\ $x \in X$, since $\lambda_x \ll \omega_p$ for a.e.\
  $x \in X$.  Then
  \begin{align*}
    f(gx) &= \{n \in \N : \lambda_{gx}(E_n)=1/2\}\\
    &= \{n \in \N : c(g,x)_*\lambda_x(E_n)=1/2\}\\
    &= f(x),
  \end{align*}
 since for any measure $\nu$ on $2^\N$ and $T \in 2^\N_{fin}$,
  $T_*\nu(E_n) = 1 - \nu(E_n)$ if $n \in T$ and $T_*\nu(E_n) =
  \nu(E_n)$ otherwise; in any case, $T_*\nu(E_n) = 1/2$ iff $\nu(E_n)
  = 1/2$.

  By the ergodicity of $\eta$ it follows that there exists a $T \in
  2^\N_{fin}$ such that $f(x) = T$ for a.e.\ $x \in X$. Let $m =
  \max(T)$, and denote $\tilde c(g,x) = c(g,x) \setminus
  \{1,\ldots,m\}$.
   
  Let $\psi(x) = \{n > m: \lambda_x(E_n) < 1/2\}$.  Then
  \begin{align*}
    \psi(gx) &= \{n > m : \lambda_{gx}(E_n) < 1/2\}\\
    &= \{n > m: c(g,x)_*\lambda_x(E_n) < 1/2\}\\
    &= \tilde c(g,x) \vartriangle \{n > m: \lambda_x(E_n) < 1/2\}\\
    &= \tilde c(g,x)\vartriangle \psi(x),
  \end{align*}
  where the third equality is a consequence of the fact that
  $\lambda_x(E_n) \neq 1/2$ for every $n > m$ and a.e.\ $x \in
  X$. Note that $\psi(x)$ is finite, by another application of
  Lemma~\ref{lemma:converse-lln}.  Hence $\tilde c$ is a coboundary,
  and so $c$ is cohomologous to the cocycle $c\vartriangle \tilde{c}$ whose image is in $2^m$.
\end{proof}

\section{Proof of main theorem}\label{sec:proof}
\begin{proof}[Proof of Theorem~\ref{thm:main}]

  By Theorem~\ref{thm:nevo} we may assume $G$ does not have
  property (T). Let $\mu$ be a generating measure on $G$ and let
  $\rho:G \to \N$ be a proper function such that
  $$\sum_{g\in G} \mu(g) \rho(g) \le 2.$$
  By Proposition~\ref{prop:key} there exists an ergodic probability
  measure preserving action $G \cc (X,\eta)$ and a cocycle $c:G \times
  X \to 2^\N_{fin}$ that is not cohomologous to a cocycle with a
  finite image. Moreover,
  $$\int_X \max(c(g,x))~d\eta(x) \le \rho(g)$$
  for any $g\in G$. By Proposition~\ref{prop:Z} for every
  $\epsilon > 0$ there exists a Bernoulli probability measure
  $\omega_{p(\epsilon)}$ on $2^\N$ such that for any probability measure
  $\nu$ on $2^\N_{fin}$,
  \begin{align*}
    h_\nu(2^\N,\omega_{p(\epsilon)}) \leq \epsilon \sum_{T \in
      2^\N_{fin}}\nu(T)\max(T).
  \end{align*}   

  Consider the skew-product action $G \cc X \times 2^\N$ given by
  $g(x,T) = (gx, c(g,x)T)$. By Lemma~\ref{lem:entropy},
  \begin{eqnarray*}
    h_\mu(X\times 2^\N, \eta \times \omega_{p(\epsilon)}) &=& \int_X
    h_{\mu_x}(2^{\N},\omega_{p(\epsilon)}) ~d\eta(x) \\
    &\le& \epsilon \int_X \sum_{T \in 2^\N_{fin}} \mu_x(T)\max(T)~d\eta(x).
  \end{eqnarray*}
  Note that, by Proposition~\ref{prop:key}, we know that $$\int_X \sum_{T
    \in 2^\N_{fin}}\mu_x(T)\max(T)~d\eta (x) = \sum_{g\in G} \mu(g) \int_X
  \max(c(g,x)) ~d\eta(x) \le \sum_{g\in G} \mu(g) \rho(g)\le 2,$$
  and so
  \begin{eqnarray*}
    \lim_{\epsilon \to 0} h_\mu(X\times 2^\N, \eta \times \omega_{p(\epsilon)}) = 0.
  \end{eqnarray*}

  Since $c$ is not cohomologous to a cocycle with a finite image,
  Lemma~\ref{lem:invariant} implies that each measure $\eta \times
  \omega_{p(\epsilon)}$ is properly nonsingular, and furthermore
  almost every measure in its ergodic decomposition is properly
  nonsingular. Since the entropy of $\eta \times \omega_{p(\epsilon)}$
  is a convex combination of the entropies of its ergodic components,
  it follows that there exist ergodic, properly nonsingular
  $G$-actions with arbitrarily small entropy, and therefore $(G,\mu)$
  does not have an entropy gap.
\end{proof}
\appendix
\section{Groups with property (T)}
\label{sec:nevo}
\begin{thm}[Nevo]
  \label{thm:nevo}
  Let $G$ be a countable group with property (T), and let
  $\mu$ be a generating measure. Then $(G,\mu)$ has an entropy gap.
\end{thm}
The following proof is based on Nevo's~\cite{Ne03}.
\begin{proof}  

  Let $\bar\mu$ be the measure on $G$ given by
  \begin{align*}
    \bar\mu = \sum_{n=0}^\infty2^{-n-1}\mu^n,
  \end{align*}
  where $\mu^n$ is the convolution of $\mu$ with itself $n$
  times. Then $h_{\bar\mu}(X,\eta) = h_\mu(X,\eta)$ (see,
  e.g.~\cite{kaimanovich1983random},
  or~\cite[Section~2.8]{hartman2012furstenberg}). The advantage of
  $\bar\mu$ is that it is supported everywhere on $G$.

  Consider, given a nonsingular action $G\cc (X,\eta)$, the unitary
  representation $\pi$ on $L^2(X,\eta)$ given by
  \begin{align*}
    \pi(g)f(x)=\sqrt{\frac{dg_*\eta}{d\eta}(x)}f(g^{-1}x).
  \end{align*}
  Note that $\pi(g)^* = \pi(g^{-1})$.
  
  It is easy to check that by Jensen's inequality, for any $g\in G$,
  \begin{align}\label{eq:Jensen-bound}
    -2\log \left\langle 1,\pi(g)1\right\rangle = -2\log
    \left\langle\pi(g^{-1})1,1\right\rangle \le \int_X
    -\log\frac{dg^{-1}_*\eta}{d\eta}(x) d\eta(x).
  \end{align}
  Consider the Markov operator $\pi(\bar\mu) \colon L^2(X,\eta)\to
  L^2(X,\eta)$ given by $\pi(\bar\mu)=\sum_{g\in G}\bar\mu(g)\pi(g)$.
  Then~\eqref{eq:Jensen-bound} yields the bound $-2\log \| \pi(\bar\mu) \|
  \le h_{\bar\mu}(X,\eta)$, by another application of Jensen's inequality.
  
  Denote
  \begin{align*}
    \|\bar\mu\|_T=\sup\left\{ \| \pi\left(\bar\mu\right)\| :\pi\mbox{ is a
        unitary representation of \ensuremath{G} with no invariant
        vectors}\right\}.
  \end{align*}
  If $G$ has property (T) then $\|\bar\mu\|_T <1$ (see, e.g., Bekka, de
  La Harpe and Valette~\cite[Corollary 6.2.3]{bekka2008kazhdan}); here
  we use the fact that $\bar\mu$ is supported everywhere.

  When $G \cc (X,\eta)$ is properly nonsingular ergodic action,
  $L^2(X,\eta)$ has no invariant vectors (see~\cite[Lemma
  7.2]{Bo14}). It therefore follows that
  \begin{align*}
    0 < -2\log\|\bar\mu\|_T \leq -2\log\|\pi(\bar\mu)\| \leq h_{\bar\mu}(X,\eta) =
    h_\mu(X,\eta),
  \end{align*}
  and $(G,\mu)$ has an entropy gap, with $\epsilon(\mu) = -2\log\|\bar\mu\|_T$.
\end{proof}

\bibliography{entropy_t}
\bibliographystyle{abbrv}

\end{document}